\documentclass{amsart}

\usepackage{amssymb,amsmath}
\usepackage[section]{algorithm}
\usepackage{algpseudocode}

\newcommand{\rst}[1]{\ensuremath{{\mathbin\mid}%
\raise-.4ex\hbox{\scriptsize$#1$}}}

\newtheorem{theorem}{Theorem}[section]
\newtheorem{proposition}[theorem]{Proposition}
\newtheorem{lemma}[theorem]{Lemma}

\newtheorem{definition}[theorem]{Definition}

\newtheorem{remark}[theorem]{Remark}

\newtheorem{example}[theorem]{Example}
\newenvironment{Example}{\begin{example}\em}{\end{example}}

\def\QQ{{\mathbb Q}}
\def\KK{{\mathbb K}}

\def\Ker{{\mathrm{Ker\,}}}
\def\Im{{\mathrm{Im\,}}}

\def\Tor{{\mathrm{Tor}}}

\def\Gr{Gr\"obner}

\def\lm{{\mathrm{lm}}}

\def\LM{{\mathrm{LM}}}

\def\bF{{\bar{F}}}
\def\bI{{\bar{I}}}
\def\bA{{\bar{A}}}

\def\bR{{\bar{R}}}
\def\bJ{{\bar{J}}}
\def\bS{{\bar{S}}}

\def\be{{\bar{e}}}
\def\bM{{\bar{M}}}
\def\bK{{\bar{K}}}
\def\bvarphi{{\bar{\varphi}}}

\def\e{{\epsilon}}
\def\bee{{\bar{\epsilon}}}

\def\bw{{\bar{w}}}

\def\FP{{\mathrm{FP}}}

\begin{document}

\title[Computing minimal free resolutions $\ldots$]
{Computing minimal free resolutions of right modules
over noncommutative algebras}

\author[R. La Scala]{Roberto La Scala$^*$}

\address{$^*$ Dipartimento di Matematica, Universit\`a di Bari, Via Orabona 4,
70125 Bari, Italy}
\email{roberto.lascala@uniba.it}

\thanks{Partially supported by Universit\`a di Bari}



\maketitle

\begin{abstract}
In this paper we propose a general method for computing a minimal free right
resolution of a finitely presented graded right module over a finitely presented
graded noncommutative algebra. In particular, if such module is the base field
of the algebra then one obtains its graded homology. The approach is based on
the possibility to obtain the resolution via the computation of syzygies for
modules over commutative algebras. The method behaves algorithmically if one
bounds the degree of the required elements in the resolution. Of course, this implies
a complete computation when the resolution is a finite one. Finally, for a monomial
right module over a monomial algebra we provide a bound for the degrees of the
non-zero Betti numbers of any single homological degree in terms of the maximal
degree of the monomial relations of the module and the algebra.
\end{abstract}



\section{Introduction}

A fundamental tool in the study of commutative and noncommutative
algebraic structures consists in developing some homology (or cohomology)
theory. In this context, the notion of free resolution, that is,
of an exact sequence of free modules over an algebra is an important one.
For instance, by means of the augmentation ideal of a graded algebra
one defines the minimal resolution of the base field which is used
to compute the graded homology of such algebra. Observe that
in the noncommutative case there are different notions of free resolution
depending on we consider one-sided or two-sided modules.
In fact, for many applications the free right (or left) resolutions are
the most effective ones and we will stick to this case in the present paper.
It is important also to mention that owing to non-Noetherianity of a general
(noncommutative finitely generated) algebra, even a finitely generated
module may have that the corresponding right syzygy module is infinitely
generated. There are of course important classes of algebras
which are right Noetherian (finite dimensional, universal enveloping
algebras, etc) or at least have finite homology. It was Anick \cite{An}
who proved that the algebras having a finite \Gr\ basis for their ideal
of relations hold finite Betti numbers for the minimal resolution
of the base field. He proved this by constructing the so-called Anick's
resolution which extends the combinatorial minimal resolution that was
introduced by Backelin \cite{Ba} for the case of monomial algebras.
One problem with the Anick's costruction is that it is generally a non-minimal
resolution and hence one needs some extra algorithmic work to minimalize it.
Beside this approach, many other ad hoc methods are used to compute
the homology (cohomology in case of group algebras) of specific classes
of algebras like modular group algebras, quiver algebras (path algebras),
G-algebras (PBW algebras, algebras of solvable type), etc. For an overview
about these methods, see for instance \cite{Gr,GS,LeS}.

In this paper we follow a different path aiming to describe a general
method for computing a minimal free right resolution of a finitely presented
graded right module over a finitely presented graded (noncommutative)
algebra by using syzygy computations for modules over commutative algebras.
In case the resolution is infinite with respect to the Betti numbers or its
length, the proposed method behaves algorithmically by bounding the degree
of the syzygies to be computed. The idea to use commutative modules has
its roots in a series of papers \cite{LSL1,LSL2,LS1} where commutative analogues
of noncommutative graded ideals has been introduced for the purpose
of obtaining noncommutative \Gr\ bases computations from commutative ones.
In those papers such correspondence between ideals is called the {\em letterplace
correspondence}. It is essentially based on the fact that a word
$w = x_{i_1}\cdots x_{i_d}$ can be represented as a set of commutative
monomials $\{x_{i_1s+1}\cdots x_{i_ds+d}\}$ ($s\geq 0$) where the index $s+k$
of the letterplace variable $x_{i_ks+k}$ corresponds to the place where
the letter $x_{i_k}$ may occur in a word containing $w$.
As a by-product one obtains, for instance, that the Hilbert function of
a noncommutative graded algebra $A$ can be obtained by computing the dimension
of suitable components of a corresponding commutative algebra. A brief
review of these ideas can be found in this paper in Section 4. The extension
of them to submodules of graded free right $A$-modules is developed
in Section 5 and 6. In Section 7 we analyze how such module letterplace
correspondence behaves with respect to homomorphisms between free modules.
This study culminates with Theorem \ref{kercorr2} which provides the new
method to obtain minimal free right resolutions. In view of this result
we propose in Section 3 a process, preserving syzygies, which allows
to reduce any grading of a free right module to the standard one where
all degrees of a free basis are equal to 1. The previous Section 2 acquaints
the reader to all noncommutative structures and basic notation that are used
in the paper. In Section 8 one finds the detailed description of the computation
of a finite minimal free right resolution of the base field of the
universal enveloping algebra of a free nilpotent Lie algebra. In Section 9
we explain how similar complete computations can be performed in an algorithmic way
by means of our method. In particular, for finitely presented monomial right modules
over finitely presented monomial algebras we propose a bound for the degrees
of the non-zero Betti numbers of any single homological degree.
In Section 10 we explain some tricks that can be used to speedup computer
calculations. By means of some of them, we propose also a comparison between
the timing obtained by an experimental application of our technique
to the example in Section 8 with the performance of an optimized library
for G-algebras. Finally, in Section 11 we present some conclusions and ideas
for further developments of the proposed approach.


\section{Right syzygy modules}

We start by introducing all noncommutative structures that we need in this paper.
Let $\KK$ be any field and denote by $F =  \KK\langle x_1,\ldots,x_n \rangle$
the free associative algebra which is freely generated by the finite set
$\{x_1,\ldots,x_n\}$. In other words, the elements of $F$ are noncommutative
polynomials in the variables $x_i$. We denote by $W\subset F$ the subset
of all monomials of $F$, that is, the elements of $W$ are words over the
alphabet $\{x_1,\ldots,x_n\}$. We endow $F$ with the {\em standard grading}
which defines $\deg(w)$ as the total degree of the monomial $w\in W$.
In other words, one has that $\deg(x_i) = 1$, for all variables $x_i$.
Denote by $F = \bigoplus_{d\geq 0} F_d$ the decomposition of the algebra $F$
in its homogeneous components. Let $I\subset F$ be a graded two-sided ideal,
that is, $I = \sum_d I_d$ where $I_d = I\cap F_d$ and consider the quotient
graded algebra $A = F/I$. Observe that if $A'$ is a graded algebra which is
finitely generated by the homogeneous set $\{x'_1,\ldots,x'_n\}$ where
$\deg(x'_i) = 1$ for all $i$ then $A'$ is clearly isomorphic to $A$
by the graded algebra homomorphism $\varphi:F\to A'$ such that $x_i\mapsto x'_i$
and $I = \Ker\varphi$. In other words, $I$ is the ideal of the relations
that are satisfied by the generators $x'_i$. By definition, the algebra $A'$
or equivalently $A$ is {\em finitely presented} when the ideal $I$ is finitely
generated. By abuse of notation, we denote also by $x_i$ the cosets $x_i + I$
which are the generators of the algebra $A$.

Let $\delta\geq 0$ be an integer and denote by $A[-\delta] =
\bigoplus_{d\geq 0} A[-\delta]_d$ the algebra $A$ which is endowed with the grading
induced by $\delta$, that is, we put $A[-\delta]_d = 0$, for any $d < \delta$
and $A[-\delta]_d = A_{d - \delta}$, for all $d\geq \delta$. Fix an integer $r > 0$
and some integers $\delta_i\geq 0$ ($1\leq i\leq r$). We consider the direct sum
$\bigoplus_{1\leq i\leq r} A[-\delta_i]$ which is a graded free right $A$-module
of finite rank $r$. If $\{e_1,\ldots,e_r\}$ is the canonical free basis of such
module then $\deg(e_i) = \delta_i$, for all $i$. A homogeneous element
of $\bigoplus_i A[-\delta_i]$ of degree $d$ is hence a right linear combination
$\sum_i e_i f_i$ ($f_i\in A$) where $\delta_i + \deg(f_i) = d$, for any $i$. 

\begin{definition}
Consider a right submodule $M\subset \bigoplus_{1\leq i\leq r} A[-\delta_i]$.
We call $M$ a {\em graded submodule} if $M = \sum_d M_d$ where $M_d =
M\cap \bigoplus_i A[-\delta_i]_d$. In this case, we define the quotient
graded right module $N = \bigoplus_{1\leq i\leq r} A[-\delta_i]/M$ where
the homogeneous component $N_d$ is isomorphic to $\bigoplus_i A[-\delta_i]_d/M_d$,
for all $d\geq 0$. A finitely generated graded right module
$N' = \langle g_1,\ldots,g_r \rangle$ where $g_i$ is a homogeneous
element of degree $\delta_i$ is clearly isomorphic to $N$ by the graded right
module homomorphism $\varphi:\bigoplus_i A[-\delta_i]\to N'$ such that
$e_i\mapsto g_i$ and $M = \Ker\varphi$. Since $M$ is the submodule of the relations
which are satisfied by the generators $g_i$ then we have that the module $N'$
or equivalently $N$ is {\em finitely presented} when $M$ is finitely generated.
\end{definition}

Remark that even if we assume that $A$ is finitely presented, this is generally
not a right Noetherian algebra which implies that a right submodule $M\subset A^r$
is usually infinitely generated. Of course, there are important cases when $A$
is right (or left) Noetherian as the case when $A$ is finite-dimensional
or it is the universal enveloping algebra of a finite-dimensional Lie algebra,
etc. (see \cite{MR}). For graded right modules one can define the notion
of minimal (homogeneous) basis and a noncommutative version of the Nakayama's
lemma holds (see, for instance, \cite{PP}). It makes sense then to define
the following invariants.

\begin{definition}
Let $A = F/I$ be a finitely generated graded algebra which is endowed with
the standard grading and consider a right submodule
$M\subset \bigoplus_{1\leq i\leq r} A[-\delta_i]$. Let $\{g_j\}$ be a
(possibly infinite) minimal basis of $M$. For all degrees $d\geq 0$,
since the homogeneous component $M_d$ is finite dimensional we can define
\[
b_d(M) = \#\{g_j\mid \deg(g_j) = d\}.
\]
The integers $\{b_d(M)\}_{d\geq 0}$ are called the {\em graded Betti numbers
of $M$}. In addition, consider a finitely generated graded right module
$N = \bigoplus_{1\leq i\leq r} A[-\delta_i]/M$. We define $b_d(N)$
as the number of generators of degree $d$ in a minimal basis of $N$
and we call again $\{b_d(N)\}_{d\geq 0}$ the {\em graded Betti numbers of $N$}.
\end{definition}

Assume now that $\{g_1,\ldots,g_s\}$ is a finite homogeneous basis of a finitely
generated graded right submodule $M\subset \bigoplus_{1\leq i\leq r} A[-\delta_i]$.
If $\Delta_j = \deg(g_j)$ then we consider the finitely generated graded free
right module $\bigoplus_{1\leq j\leq s} A[-\Delta_j]$. By denoting $\{\e_1,\ldots,\e_s\}$
its canonical basis we have therefore that $\deg(\e_j) = \Delta_j$.
It is natural to define the graded right module homomorphism
\[
\varphi:\bigoplus_{1\leq j\leq s} A[-\Delta_j]\to
\bigoplus_{1\leq i\leq r} A[-\delta_i]\,,\, \e_j\mapsto g_j.
\]
Its image is clearly $\Im\varphi = M$ and its kernel is by definition
\[
K = \Ker\varphi = \{\sum_j\e_j f_j\mid f_j\in A, \sum_j g_j f_j = 0\}.
\]
We call the elements of the graded right submodule $K\subset
\bigoplus_{1\leq j\leq s} A[-\Delta_j]$ the {\em right syzygies of the basis $\{g_j\}$}.
We refer also to $K$ as the {\em right syzygy module of $M$} (with respect to $\{g_j\}$).
Note that possible synonyms for the term ``right syzygies'' are {\em right module
relations} or {\em right linear relations}.

The main concern of the present paper is to develop a method to compute a homogeneous
basis of $K$. We have to remark immediately that the right syzygy module $K$
is not necessarily finitely generated again because $A$ is generally not a right
Noetherian algebra or at least a right coherent algebra \cite{Pi}.
Nevertheless, there are many cases when $K$ is finitely generated that we will
discuss in Section 9. For the moment, let us just observe that if $M,K$
are not finitely generated then one can at least study homogeneous syzygies of
homogeneous generators up to some fixed degree. A natural generalization
of the notion of right syzygy module is the following one.

\begin{definition}
Let $N = \bigoplus_{1\leq i\leq r} A[-\delta_i]/M$ be a finitely generated graded
right module. A {\em graded free right resolution of $N$} is by definition a
(possibly infinite) exact sequence of graded right module homomorphisms
\[
0 \leftarrow N\leftarrow \bigoplus_{1\leq i\leq r} A[-\delta_i] \leftarrow
\bigoplus_{1\leq j\leq s} A[-\delta'_j] \leftarrow
\bigoplus_{1\leq k\leq t} A[-\delta''_k] \leftarrow \ldots.
\]
Denote by $M_i$ the kernel of the $(i+1)$-th map of the above sequence (including
the first zero map). Note that $M_0 = N$ and $M_1 = M$. We call $M_i$ the
{\em $i$-th right syzygy module of $N$ (with respect to the resolution)}. We call
this sequence a {\em minimal resolution of $N$} if all the finite homogeneous bases
of $M_i$ which are the images of the canonical bases of the free modules are minimal
ones. In this case, we refer to the integer $b_d(M_i)$ as the {\em graded Betti
number of $N$ of homological degree $i$ and (internal) degree $d$}.
\end{definition}

It may happen that a resolution is finite because the second to last map
is an injective one. In this case, one says that $N$ admits a {\em finite free
right resolution}. For instance, the right modules over a G-algebra hold
such resolutions \cite{LeS}. Another interesting case is when there is a resolution
with exactly $m$ right syzygy modules (including $M_0 = N$) which are
finitely presented, that is, $M_m$ is finitely generated but infinitely
related. In this case the right module $N$ is called {\em of type $(\FP)_m$}.
Finally, if all resolutions have infinite length then $N$ is by definition
{\em of type $(\FP)_\infty$}. Observe that the modules over Noetherian
commutative algebras either have finite resolutions or belong to type $(\FP)_\infty$
(see, for instance, \cite{Pe}).

\section{Module component homogenization}

We introduce now a construction which provides that for the computation of the right
syzygies of a graded right submodule $M\subset\bigoplus_{1\leq i\leq r} A[-\delta_i]$
one is always reduced to the case of the standard grading, that is, $\delta_i = 0$
for any $i$ and hence $M\subset A^r$. This is in fact an essential step in view
of the method that we will propose in Section 7.

Let $t\notin F = \KK\langle x_1,\ldots,x_n \rangle$ be a new variable and consider 
the free associative algebra $\bF = \KK\langle x_1,\ldots,x_n, t \rangle$ endowed
with the standard grading. If $A = F/I$ where $I\subset F$ is a graded two-sided ideal
then we denote by $\bI$ the extension of $I$ in $\bF$, that is, $\bI\subset\bF$
is the graded two-sided ideal generated by $I$. Then, we define the quotient graded
algebra $\bA = \bF/\bI$. Since $I = \bI\cap F$, one has that $A$ can be canonically
embedded in $\bA$.

Fix now two integers $\delta\geq\delta'\geq 0$. We have an injective graded right
$A$-module homomorphism $\eta:A[-\delta]\to \bA[-\delta']$ which is defined as
$1\mapsto t^{\delta-\delta'}$. This map can be extended to graded right free modules
in the following way. Consider the integers $\delta_i\geq\delta'_i\geq 0$
($1\leq i\leq r$) and let $\{e_i\},\{\be_i\}$ be the canonical bases of the
graded free right modules $\bigoplus_{1\leq i\leq r} A[-\delta_i],
\bigoplus_{1\leq i\leq r} \bA[-\delta'_i]$ respectively. We define the injective
graded right $A$-module homomorphism
\[
\eta:\bigoplus_i A[-\delta_i]\to \bigoplus_i \bA[-\delta'_i]\,,\,
e_i\mapsto \be_i t^{\delta_i-\delta'_i}.
\]

\begin{definition}
Let $M\subset \bigoplus_i A[-\delta_i]$ be a graded right $A$-submodule.
Denote by $H(M)\subset\bigoplus_i \bA[-\delta'_i]$ the graded right $\bA$-submodule
generated by $\eta(M)$. We call $H(M)$ a {\em (partial) component homogenization of $M$}.
In particular, if $\delta'_i = 0$ for any $i$, that is, $\bigoplus_i \bA[-\delta'_i] = \bA^r$
with the standard grading then we say that $H(M)$ is the {\em complete component
homogenization of $M$}. On the other hand, if $\delta'_i = \delta_i$ for all $i$, that is,
$H(M)$ is the extension of an $A$-submodule of $\bigoplus_i A[-\delta_i]$ to a
$\bA$-submodule of $\bigoplus_i \bA[-\delta_i]$ then we call $H(M)$ the {\em trivial
component homogenization}.
\end{definition}

It is clear that $\eta(\bigoplus_i A[-\delta_i]) =
\bigoplus_i t^{\delta_i-\delta'_i} A[-\delta_i]$ which is a graded right $A$-submodule
of $\bigoplus_i \bA[-\delta'_i]$. This result can be generalized in the following way.

\begin{proposition}
Let $M\subset \bigoplus_i A[-\delta_i]$ be a graded right submodule and denote
$\bM = H(M)$ and $\bM'= \bM\cap \bigoplus_i t^{\delta_i-\delta'_i} A[-\delta_i]$.
Then, one has that $\eta(M) = \bM'$. In particular, we have that $\dim_\KK M_d =
\dim_\KK \bM'_d$, for all $d\geq 0$.
\end{proposition}

\begin{proof}
Since $\eta(\bigoplus_i A[-\delta_i]) = \bigoplus_i t^{\delta_i-\delta'_i} A[-\delta_i]$,
it is clear that $\eta(M)\subset \bM'$. Moreover, any element $h'\in\bM$ is such that
$h' = \sum_j g'_j f'_j$ where $g'_j = \eta(g_j), g_j\in M$ and $f'_j\in\bA$. In other words,
if $h' = \sum_i \be_i h'_i$ ($h'_i\in\bA$) and $g_j = \sum_i e_i g_{ij}$ ($g_{ij}\in A$)
then $h'_i = t^{\delta_i-\delta'_i} \sum_j g_{ij} f'_j$. By assuming that $h'\in \bM'$,
that is, $h'_i\in t^{\delta_i-\delta'_i} A$ we obtain clearly that $f'_j\in A$ and
therefore $h' = \eta(h)$ where $h = \sum_j g_j f'_j\in M$.
\end{proof}

Since $\eta$ is an injective map, note that the above result implies that the mapping
$M\mapsto H(M)$ is also an embedding of the graded right $A$-submodules
of $\bigoplus_i A[-\delta_i]$ into the graded right $\bA$-submodules of
$\bigoplus_i \bA[-\delta'_i]$. It is clear that the component homogenizations
are exactly those graded right $\bA$-submodules $\bM\subset\bigoplus_i \bA[-\delta'_i]$
that are generated by $\bM' = \bM\cap\bigoplus_i t^{\delta_i-\delta'_i} A[-\delta_i]$.
Moreover, since $\eta$ is an injective graded right $A$-module homomorphism, we have that
$M = \eta^{-1}(\bM')\subset\bigoplus_i A[-\delta_i]$ is the unique graded right
$A$-submodule such that $H(M) = \bM$. We want now to study how the correspondence
$M\mapsto H(M)$ behaves with respect to minimal generating sets.

\begin{theorem}
\label{bascorr}
Consider a graded right submodule $M\subset \bigoplus_i A[-\delta_i]$ and its
component homogenization $\bM = H(M)\subset \bigoplus_i \bA[-\delta'_i]$. Moreover,
let $\{g_j\}$ be a set of homogeneous elements of $M$ and define $g'_j = \eta(g_j)$,
for all $j$. Then, we have that $\{g_j\}$ is a (minimal) basis of $M$
if and only if $\{g'_j\}$ is a (minimal) basis of $\bM$. In particular, one has
that $b_d(M) = b_d(\bM)$, for all $d\geq 0$.
\end{theorem}

\begin{proof}
Assume that $\{g_j\}$ is a homogeneous basis of $M$. Since $\eta$ is a graded $A$-module
homomorphism, one has that $\{g'_j\}$ is a homogeneous basis of the graded $A$-submodule
$\bM' = \bM\cap \bigoplus_i t^{\delta_i-\delta'_i} A[-\delta_i]$. Because the graded
$\bA$-submodule $\bM$ is generated by $\bM'$ we obtain immediately that $\{g'_j\}$
is a homogeneous basis of $\bM$.

Suppose now that $g'_1 = \sum_{j>1} g'_j f'_j$ where $f'_j\in\bA$, that is, $\{g'_j\}$
is not a minimal basis. If $g_j = \sum_i e_i g_{ij}$ ($g_{ij}\in A$) then one has that
$t^{\delta_i-\delta'_i} g_{i1} = \sum_{j>1} t^{\delta_i-\delta'_i} g_{ij} f'_j$.
We conclude that $f'_j\in A$ and $g_1 = \sum_{j>1} g_j f'_j$. Finally, it is clear that
by similar arguments one obtains also the necessary condition in the statement.
\end{proof}

Note explicitly that the above result provides a method to obtain a minimal
basis $\{g_j\}$ of the graded submodule $M$ starting from a minimal basis $\{g'_j\}$
of its component homogenization $\bM = H(M)$. In fact, since $\bM$ is generated by
$\bM' = \eta(M)$, we can assume that $\{g'_j\}\subset\bM'$ and hence $g_j = \eta^{-1}(g'_j)$,
for all $j$. In Section 2 we have already observed that we cannot always assume
that these bases are finite sets but in what follows we will make this assumption.

Let $\{g_j\}$ be a finite homogeneous basis of a finitely generated graded right 
$A$-submodule $M\subset \bigoplus_i A[-\delta_i]$. Then, denote $\Delta_j = \deg(g_j)$
and consider the finitely generated graded free right $A$-module $\bigoplus_j A[-\Delta_j]$.
If $\{\e_j\}$ is its canonical basis then by definition we have that $\deg(\e_j) = \Delta_j$,
for any $j$. In a similar way, one defines the finitely generated graded free right
$\bA$-module $\bigoplus_j \bA[-\Delta_j]$ with canonical basis $\{\bee_j\}$ such that
$\deg(\bee_j) = \Delta_j$, for all $j$. One has therefore the graded right $A$-module
homomorphism
\[
\varphi:\bigoplus_j A[-\Delta_j]\to \bigoplus_i A[-\delta_i]\,,\, \e_j\mapsto g_j.
\]
By putting $g'_j = \eta(g_j)$, we consider also the graded right $\bA$-module homomorphism
 \[
H(\varphi):\bigoplus_j \bA[-\Delta_j]\to \bigoplus_i \bA[-\delta'_i]\,,\, \bee_j\mapsto g'_j.
\]
Observe that by Theorem \ref{bascorr} we have that $\Im H(\varphi) = H(\Im\varphi) =
H(M)$. We want now to understand how the correspondence $\varphi\mapsto H(\varphi)$ behaves
with respect to the kernels. By denoting $\bvarphi = H(\varphi)$, we have by definition
\begin{equation*}
\begin{gathered}
K = \Ker\varphi = \{\sum_j\e_j f_j\mid f_j\in A, \sum_j g_j f_j = 0\}, \\
\bK = \Ker\bvarphi = \{\sum_j\bee_j f'_j\mid f'_j\in\bA, \sum_j g'_j f'_j = 0\}.
\end{gathered}
\end{equation*}
In other words, the graded right $A$-submodule $K\subset \bigoplus_j A[-\Delta_j]$
is the right syzygy module of the basis $\{g_j\}$ and the graded right $\bA$-submodule
$\bK\subset \bigoplus_j \bA[-\Delta_j]$ is the right syzygy module of $\{g'_j\}$.

Consider now the canonical embedding $\bigoplus_j A[-\Delta_j]\to \bigoplus_j \bA[-\Delta_j]$
such that $\e_j\mapsto\bee_j$. By abuse of notation, we denote this injective graded right
$A$-module homomorphism also by $\eta$ since it defines the trivial component homogenization
$H(K)\subset \bigoplus_j \bA[-\Delta_j]$ of the submodule $K\subset \bigoplus_j A[-\Delta_j]$.
In fact, we will prove that $\bK = H(K)$, that is, $\Ker H(\varphi) = H(\Ker\varphi)$.
We start by considering the following key property.

\begin{proposition}
\label{commdia}
The following diagram is a commutative one
\[
\begin{array}{ccc}
\bigoplus_j A[-\Delta_j] & {\buildrel \varphi\over\longrightarrow} & 
\bigoplus_i A[-\delta_i] \\
\eta \downarrow & & \hspace{9pt} \downarrow \eta \\
\bigoplus_j \bA[-\Delta_j] & {\buildrel H(\varphi)\over\longrightarrow} & 
\bigoplus_i \bA[-\delta'_i] \\
\end{array}
\]
\end{proposition}

\begin{proof}
Put $\bvarphi = H(\varphi)$. For all $h = \sum_j \e_j f_j\in \bigoplus_j A[-\Delta_j]$,
we have to prove that $\eta(\varphi(h)) = \bvarphi(\eta(h))$. In fact, one has that
$\eta(\varphi(h)) = \eta(\sum_j g_j f_j) = \sum_j \eta(g_j) f_j = \sum_j g'_j f_j$.
On the other hand, it holds that $\bvarphi(\eta(h)) = \bvarphi(\sum_j \bee_j f_j) =
\sum_j g'_j f_j$.
\end{proof}

\begin{theorem}
\label{kercorr}
One has that $\bK = H(K)$, that is, $\eta(K) = \bK' = \bK\cap \bigoplus_j A[-\Delta_j]$.
\end{theorem}

\begin{proof}
By Proposition \ref{commdia}, we have immediately that $\eta(K)\subset \bK$
and therefore $H(K)\subset \bK$. It is sufficient to prove that $\bK$ has
a generating set which is contained in $\eta(K)$. Consider $h' = \sum_j \bee_j f'_j\in\bK$
any element of a minimal basis of $\bK$ where $f'_j\in \bA_{d-\Delta_j}$, for some $d\geq 0$
and for all $j$. If $g_j = \sum_i e_i g_{ij}$ where $g_{ij}\in A_{\Delta_j-\delta_i}$
then $g'_j = \sum_i \be_i t^{\delta_i} g_{ij}$ and hence in the algebra $\bA$
we have that $t^{\delta_i} \sum_j g_{ij} f'_j = 0$ for all $i$. Recall now 
the graded algebras $A = F/I$ and $\bA = \bF/\bI$ are defined by the graded two-sided
ideal $I\subset F$ and $\bI\subset\bF$ where $\bI$ is generated by $I = \bI\cap F$.
Since $g_{ij}\in A$, we obtain therefore that $\sum_j g_{ij} f'_j = 0$ for any $i$.
Moreover, because $h'$ belongs to a minimal basis one has also that $f'_j\in A$
for all $j$, that is, $h' = \eta(h)$ where $h = \sum_j \e_j f'_j\in K$.
\end{proof}

Observe explicitly that Theorem \ref{bascorr} together with Theorem \ref{kercorr}
provides a method for obtaining the couple of graded Betti numbers sets $\{b_d(M)\}$
and $\{b_d(K)\}$ of a graded right submodule $M\subset \bigoplus_i A[-\delta_i]$
by its component homogenization $H(M)\subset \bigoplus_i \bA[-\delta'_i]$.
In the next sections we will show that this is essential to obtain these data
via commutative analogues of such structures.


\section{Letterplace ideals}

In this section we show that graded two-sided ideals and hence noncommutative
graded algebras have useful commutative counterparts. The results of this section
has been introduced \cite{LSL1,LS1} but we recall them here for the sake of completeness.

Consider the polynomial algebra $P = \KK[x_{ij}\mid 1\leq i\leq n, j\geq 1]$ 
in the infinite set of commutative variables $\{x_{ij}\}$. We assume that $P$
is endowed with the standard grading, that is, $\deg(x_{ij}) = 1$, for all $i,j$.
Then, we define the monomial ideal
\[
Q = \langle x_{ij}x_{kj}\mid 1\leq i,k\leq n, j\geq 1 \rangle\subset P
\]
and we consider the quotient graded algebra $R = P/Q$. A $\KK$-linear basis of the
algebra $R$ is given by (the cosets of) the monomials $x_{i_1j_1}\cdots x_{i_dj_d}$
($d\geq 0$) where $1\leq i_1,\ldots,i_d\leq n$ and $1\leq j_1 < \ldots < j_d$. 
If, as usual, $F = \KK\langle x_1,\ldots,x_n \rangle$ then a graded $\KK$-linear
embedding $\iota:F\to R$ is defined as
\[
x_{i_1}\cdots x_{i_d}\mapsto x_{i_11}\cdots x_{i_dd}
\]
where $x_{i_1}\cdots x_{i_d}$ is any monomial of $W$. For any degree $d\geq 0$,
the image $\iota(F_d)\subset R_d$ can be described in the following way.
Consider the finitely generated graded subalgebra
$P(d) = \KK[x_{ij}\mid 1\leq i\leq n, 1\leq j\leq d]\subset P$ and denote
$Q(d) = Q\cap P(d)$. Then, we define the quotient graded algebra $R(d) = P(d)/Q(d)$
which can be canonically embedded in $R$. One has immediately that
$\iota(F_d) = R(d)_d$.

Consider now $\sigma:R\to R$ the injective algebra endomorphism that is defined as
$x_{ij}\mapsto x_{ij+1}$, for all $i,j$. For the mapping $\iota$ one has the following
key property. For all $g\in F_d$ and $f\in F$, we have that
\[
\iota(g f) = \iota(g) \sigma^d(\iota(f)).
\]

\begin{definition}
Let $J$ be an ideal of $R$. If $\sigma(J)\subset J$ then we call $J$ a {\em $\sigma$-invariant
ideal}. In this case, if there is a subset $G\subset J$ such that $J$ is generated by
$\bigcup_{k\geq 0} \sigma^k(G)$ as an ideal of $R$ then we say that $J$ is
{\em $\sigma$-generated by $G$}.
\end{definition}

\begin{definition}
Let $I\subset F$ be a graded two-sided ideal. We define $L(I)\subset R$ as the graded
$\sigma$-invariant ideal which is $\sigma$-generated by $\iota(I)$. We call $L(I)$
the {\em letterplace analogue of $I$}.
\end{definition}

It is clear that $L(F) = R$ and we have that $\iota(F_d) = R(d)_d$, for all $d\geq 0$.
This property can be extended to any graded two-sided ideal of $F$ in the following way.

\begin{proposition}
Let $I\subset F$ be a graded two-sided ideal and denote $J = L(I)$. For any degree $d\geq 0$,
one has that $\iota(I_d) = J(d)_d$ where $J(d) = J\cap R(d)$. In particular, we obtain that
$\dim_\KK I_d = \dim_\KK J(d)_d$.
\end{proposition}

\begin{proof}
It is clear that $\iota(I_d)\subset J(d)_d$. Let now $g_j\in I$ ($1\leq j\leq s$)
be a homogeneous element of degree $d_j$ and put $g'_j = \iota(g_j)\in J$.
Any element of the subspace $J_d\subset J$ is a $\KK$-linear combination of elements
of type $h' = \sum_j \sigma^{k_j}(g'_j) m'_j$ where $m'_j$ is a monomial in $R_{d-d_j}$.
Moreover, all monomials of $\sigma^{k_j}(g'_j)$ are of type
\[
x_{i_1 k_j+1}\cdots x_{i_{d_j} k_j+d_j}
\]
with $1\leq i_1,\ldots,i_{d_j}\leq n$. Assume now that $h'\in J(d)_d = J_d\cap R(d)$.
Since by definition $R = P/Q$, it follows necessarily that 
$m'_j = u'_j\sigma^{k_j+d_j}(v'_j)$ where $u'_j$ is a monomial in $R(k_j)_{k_j}$
and $v'_j$ is a monomial in $R(d-k_j-d_j)_{d-k_j-d_j}$. By defining $u_j =
\iota^{-1}(u'_j), v_j = \iota^{-1}(v'_j)$, we conclude that
\[
h' = \sum_j \iota(u_j) \sigma^{k_j}(\iota(g_j)) \sigma^{k_j+d_j}(\iota(v_j)) =
\iota(\sum_j u_j g_j v_j)\in\iota(I_d).
\]
\end{proof}

Since $\iota$ is an injective map, the above result implies that the mapping
$I\mapsto L(I)$ is also an embedding that we call the {\em (ideal) letterplace
correspondence}. It is easy to characterize the letterplace analogues among
all graded $\sigma$-invariant ideals of $R$.

\begin{proposition}
Let $J$ be a graded $\sigma$-invariant ideal of $R$ which is $\sigma$-generated
by the graded subspace $\bigoplus_d J(d)_d$ where $J(d) = J\cap R(d)$. Then,
there exists a (unique) graded two-sided ideal $I\subset F$ such that $J = L(I)$.
\end{proposition}

\begin{proof}
For all $d\geq 0$, denote $I_d = \iota^{-1}(J(d)_d)$ and define $I = \bigoplus_d I_d$.
We have to show that $I\subset F$ is a two-sided ideal. Then, consider $f\in F_k, g\in I_d$
and $h\in F_{d'}$. One has clearly that
\[
\iota(f g h) = \iota(f)\sigma^k(\iota(g))\sigma^{k+d}(\iota(h))\in J(k+d+d')_{k+d+d'}
\]
and we conclude that $f g h\in I_{k+d+d'}$.
\end{proof}

Owing to the above result, we call {\em letterplace ideals} all the graded
$\sigma$-invariant ideals $J\subset R$ that are $\sigma$-generated by
$\bigoplus_d J(d)_d$. In other words, the letterplace correspondence establishes
a bijection between all the graded two-sided ideals of $F$ and the class of their
analogues which are the letterplace ideals of $R$.


\section{Module letterplace embedding}

As in Section 2, we consider a finitely generated graded algebra $A = F/I$ where
$I\subset F$ is a graded two-sided ideal. Moreover, with the notation of Section 4
we have the (infinitely generated commutative) graded algebra $R = P/Q$ which is
endowed with the algebra endomorphism $\sigma:x_{ij}\mapsto x_{ij+1}$.
Then, we consider $J = L(I)\subset R$ the letterplace analogue of $I$ and
we define the quotient graded algebra $S = R/J$. Since $J$ is the letterplace
analogue of $I$, observe that $S$ is endowed with an algebra endomorphism
induced by $\sigma$. By abuse of notation, we will denote this map also by $\sigma$.
Moreover, we have a graded $\KK$-linear embedding $A\to S$ induced by $\iota$
that we will denote again by this symbol.

Fix now an integer $\delta\geq 0$ and denote by $A[-\delta]$ the algebra $A$ which
is endowed with the grading induced by $\delta$. In the same way, we define the graded
algebra $S[-\delta]$. Consider now the graded $\KK$-linear embedding
$A[-\delta]\to S[-\delta]$ such that $f\mapsto \sigma^{\delta}(\iota(f))$.
By abuse of notation, we denote this map also by $\iota$. 

To the aim of describing the image $\iota(A[-\delta]_d)\subset S[-\delta]_d$
for any degree $d\geq 0$, we introduce the following objects. Consider the subalgebra
\[
P(\delta,d) = \KK[x_{ij}\mid 1\leq i\leq n, \delta + 1\leq j\leq d]\subset P
\]
and define its ideal $Q(\delta,d) = Q\cap P(\delta,d)$. Then, we consider the quotient
algebra $R(\delta,d) = P(\delta,d)/Q(\delta,d)$ which is canonically embedded in $R$.
Finally, consider the ideal $J(\delta,d) = J\cap R(\delta,d)$ and define the quotient
$S(\delta,d) = R(\delta,d)/J(\delta,d)$. If we use the grading induced by $\delta$ then
we denote this algebra as $S(\delta,d)[-\delta]$. It is immediate to show that
\[
\iota(A[-\delta]_d) = S(\delta,d)[-\delta]_d.
\]

One can easily extend the above map $\iota$ to graded free modules in the following way.
Consider the graded free right $A$-module $\bigoplus_{1\leq i\leq r} A[-\delta_i]$
whose grading is defined by the integers $\delta_i\geq 0$. Denote $\{e_i\}$ its
canonical basis and therefore one has that $\deg(e_i) = \delta_i$, for all $i$.
In a similar way, we define $\bigoplus_{1\leq i\leq r} S[-\delta_i]$ as the graded
free $S$-module such that the elements of its canonical basis $\{e'_i\}$ have degrees
$\deg(e'_i) = \delta_i$, for any $i$. Then, we consider the graded $\KK$-linear embedding
$\bigoplus_i A[-\delta_i]\to \bigoplus_i S[-\delta_i]$ such that, for any $f\in A$
\[
e_i f\mapsto e'_i \sigma^{\delta_i}(\iota(f)).
\]
By abuse of notation, we denote also this mapping as $\iota$ and we call it the
{\em module letterplace embedding}. For all degrees $d\geq 0$, we have clearly that 
\[
\iota(\bigoplus_i A[-\delta_i]_d) = \bigoplus_i S(\delta_i,d)[-\delta_i]_d.
\]
In particular, if $\delta_i = 0$ for all $i$ then one has that
$\bigoplus_i A[-\delta_i] = A^r, \bigoplus_i S[-\delta_i] = S^r$ and we denote
$S(d)^r = S(0,d)^r$, for any $d$. It holds hence that $\iota(A^r_d)\subset S(d)^r_d$.
We will make use of this notation in Section 7. We conclude this section with
the following key property.

\begin{proposition}
\label{iotapro}
Let $g\in\bigoplus_i A[-\delta_i]$ be a homogeneous element of degree $d$ and let $f\in A$.
Then, one has that $\iota(g f) = \iota(g)\sigma^d(\iota(f))$.
\end{proposition}

\begin{proof}
Denote $g = \sum_i e_i g_i$ where $g_i\in A[-\delta_i]_d = A_{d-\delta_i}$. Then
$g f = \sum_i e_i g_i f$ and therefore
\begin{equation*}
\begin{gathered}
\iota(g f) = \sum_i e'_i \sigma^{\delta_i}(\iota(g_i f)) =
\sum_i e'_i \sigma^{\delta_i} ( \iota(g_i) \sigma^{d-\delta_i}(\iota(f)) ) = \\
\sum_i e'_i \sigma^{\delta_i}(\iota(g_i)) \sigma^d(\iota(f)) =
\iota(g) \sigma^d(\iota(f)).
\end{gathered}
\end{equation*}
\end{proof}


\section{Letterplace modules}

Similarly to what we have done for ideals in Section 4, we introduce here
the concept that a noncommutative module can be associated to a commutative one
in a very meaningful way. One main difference with the case of graded two-sided
ideals is that for graded right modules we do not need to require the
$\sigma$-invariance of the commutative analogues. In fact, in the letterplace
correspondence one has that this property is inherently associated to left
structures. Note that this is a good motivation for us to prefer noncommutative
right modules to left ones.

\begin{definition}
Let $M$ be a graded right submodule of $\bigoplus_{1\leq i\leq r} A[-\delta_i]$.
We define $L(M)\subset\bigoplus_{1\leq i\leq r} S[-\delta_i]$ as the graded
submodule generated by $\iota(M)$. We call $L(M)$ the {\em letterplace analogue of $M$}.
\end{definition}

It is clear that $L(\bigoplus_i A[-\delta_i])$ is exactly the free module
$\bigoplus_i S[-\delta_i]$. Recall also that $\iota(\bigoplus_i A[-\delta_i]_d) =
\bigoplus_i S(\delta_i,d)[-\delta_i]_d$, for all degrees $d\geq 0$. This property
can be extended to any graded right submodule of $\bigoplus_i A[-\delta_i]$
in the following way.

\begin{proposition}
Let $M\subset\bigoplus_i A[-\delta_i]$ be a graded right submodule and denote
$M' = L(M)$. For any $d\geq 0$, one has that $\iota(M_d) = M'(d)_d$ where
$M'(d) = M'\cap \bigoplus_i S(\delta_i,d)[-\delta_i]$. In particular, we obtain
that $\dim_\KK M_d = \dim_\KK M'(d)_d$.
\end{proposition}

\begin{proof}
From $\iota(\bigoplus_i A[-\delta_i]_d) = \bigoplus_i S(\delta_i,d)[-\delta_i]_d$
it follows immediately that $\iota(M_d)\subset M'(d)_d$. Moreover, any element
$h'\in M'_d$ is such that $h' = \sum_j g'_j f'_j$ where $g'_j = \iota(g_j),
g_j\in M_{\Delta_j}$ and $f'_j\in S_{d-\Delta_j}$. In other words,
if $h' = \sum_i e'_i h'_i$ ($h'_i\in S_{d-\delta_i}$) and $g_j = \sum_i e_i g_{ij}$
($g_{ij}\in A_{\Delta_j-\delta_i}$) then $h'_i = \sum_j \sigma^{\delta_i}(\iota(g_{ij})) f'_j$.
Observe that $\sigma^{\delta_i}(\iota(g_{ij}))\in S(\delta_i,\Delta_j)_{\Delta_j-\delta_i}$.
By assuming that $h'\in M'(d)_d$, that is, $h'_i\in S(\delta_i,d)_{d-\delta_i}$
we obtain clearly that $f'_j\in S(\Delta_j,d)_{d-\Delta_j}$, that is,
$f'_j = \sigma^{\Delta_j}(\iota(f_j))$ with $f_j\in A_{d-\Delta_j}$. We conclude that
$h' = \iota(h)$ where $h = \sum_j g_j f_j\in M_d$.
\end{proof}

Owing to $\iota$ is an injective map, by the above result we have that the mapping
$M\mapsto L(M)$ is also an embedding that we call the {\em (module) letterplace
correspondence}. We want now to characterize the letterplace analogues among
all graded submodules of $\bigoplus_i S[-\delta_i]$.

\begin{proposition}
Let $M'$ be a graded submodule of $\bigoplus_i S[-\delta_i]$ which is generated
by the graded subspace $\bigoplus_d M'(d)_d$ where $M'(d) =
M'\cap \bigoplus_i S(\delta_i,d)[-\delta_i]$. Then, there exists a (unique) graded
right submodule $M\subset \bigoplus_i A[-\delta_i]$ such that $M' = L(M)$.
\end{proposition}

\begin{proof}
For all $d\geq 0$, denote $M_d = \iota^{-1}(M'(d)_d)$ and define $M = \bigoplus_d M_d$.
We have to show that $M$ is a right $A$-module. Then, consider $g\in M_d$ and $f\in A_{d'}$. 
One has clearly that $\iota(g f) = \iota(g)\sigma^d(\iota(f))\in M'(d+d')_{d+d'}$
and we conclude that $g f\in M_{d+d'}$.
\end{proof}

Accordingly with this result, we call {\em letterplace submodules} all the graded
submodules $M'\subset\bigoplus_i S[-\delta_i]$ that are generated by $\bigoplus_d M'(d)_d$.
Then, the letterplace correspondence defines a bijection between all the graded
right submodules of $\bigoplus_i A[-\delta_i]$ and the class of their analogues
which are the letterplace submodules of $\bigoplus_i S[-\delta_i]$.

In a similar way to the case of graded right $A$-modules, we have that minimal bases
and graded Betti numbers are defined for graded $S$-modules. We analyze now how
the correspondence $M\mapsto L(M)$ behaves with respect to such objects.

\begin{theorem}
\label{bascorr2}
Consider a graded right submodule $M\subset \bigoplus_i A[-\delta_i]$ and its
letterplace analogue $M' = L(M)\subset \bigoplus_i S[-\delta_i]$. Moreover,
let $\{g_j\}$ be a set of homogeneous elements of $M$ and define $g'_j = \iota(g_j)$,
for all $j$. Then, $\{g_j\}$ is a (minimal) basis of $M$ if and only if
$\{g'_j\}$ is a (minimal) basis of $M'$. In particular, we have that
$b_d(M) = b_d(M')$, for all $d\geq 0$.
\end{theorem}

\begin{proof}
Assume that $\{g_j\}$ is a homogeneous basis of $M$. Since $M'$ is generated by
$\iota(M) = \bigoplus_d M'(d)_d$, to prove that $\{g'_j\}$ is a basis of $M'$
it is sufficient to show that any homogeneous element $h'\in M'(d)_d$ is generated
by $\{g'_j\}$. Denote $h = \iota^{-1}(h')\in M_d$ and hence $h = \sum_j g_j f_j$
where $f_j\in A_{d - \Delta_j}$ ($\Delta_j = \deg(g_j)$). Then, we have immediately
that $h' = \sum_j g'_j \sigma^{\Delta_j}(\iota(f_j))$.

Suppose now that $g'_1 = \sum_{j>1} g'_j f'_j$ with $f'_j\in S_{\Delta_1-\Delta_j}$.
Because $g'_j = \iota(g_j)\in\bigoplus_i S(\delta_i,\Delta_j)[-\delta_i]_{\Delta_j}$,
we have necessarily that $f'_j = \sigma^{\Delta_j}(f''_j)$ for some element
$f''_j\in S(\Delta_1-\Delta_j)_{\Delta_1-\Delta_j}$. By defining $f_j = \iota^{-1}(f''_j)$,
we have therefore that $g_1 = \sum_{j>1} g_j f_j$.
Finally, by similar arguments one proves that if $\{g'_j\}$ is a (minimal) basis of $M'$
then $\{g_j\}$ is a (minimal) basis of $M$. 
\end{proof}

Observe that by Theorem \ref{bascorr2} one has a method to obtain a minimal basis
$\{g_j\}$ of the graded submodule $M$ starting from a minimal basis $\{g'_j\}$
of its letterplace analogue $M' = L(M)$. In fact, since $M'$ is generated by
$\bigoplus_d M'(d)_d = \iota(M)$, one can assume that $\{g'_j\}\subset\bigoplus_d M'(d)_d$
and therefore $g_j = \iota^{-1}(g'_j)$, for all $j$. Note finally that such bases
are not necessarily finite because both the finitely generated noncommutative
algebra $A$ and the infinitely generated commutative algebra $S$ are generally
not Noetherian ones. Of course, one can essentially restore finiteness just by
considering the homogeneous generators up to some degree.


\section{Letterplace morphisms}

Owing to the module component homogenization of Section 3, when considering submodules
of a finitely generated graded free right $A$-module we can always assume that this is
$A^r$ endowed with the standard grading. Such assumption will be in fact necessary
to the main result of this section which is Theorem \ref{kercorr2}. Consider $M\subset A^r$
a finitely generated graded right submodule and let $\{g_j\}$ ($1\leq j\leq s$)
be a minimal basis of $M$. By putting $\Delta_j = \deg(g_j)$ for all $j$, we define
the finitely generated graded free right $A$-module $\bigoplus_{1\leq j\leq s} A[-\Delta_j]$.
In other words, if $\{\e_j\}$ is its canonical basis then by definition $\deg(\e_j) =
\Delta_j$. In a similar way, one defines $\bigoplus_{1\leq j\leq s} S[-\Delta_j]$
as the finitely generated graded free $S$-module such that the elements of its canonical
basis $\{\e'_j\}$ have degrees $\deg(\e'_j) = \Delta_j$. One defines therefore
the graded right $A$-module homomorphism
\[
\varphi:\bigoplus_j A[-\Delta_j]\to A^r\,,\, \e_j\mapsto g_j
\]
such that $\Im\varphi = M$. By putting $g'_j = \iota(g_j)$, we consider also the graded
$S$-module homomorphism
\[
L(\varphi):\bigoplus_j S[-\Delta_j]\to S^r\,,\, \e'_j\mapsto g'_j.
\]
Note that Theorem \ref{bascorr2} implies that $\Im L(\varphi) = L(M)$. A critical point
is now to understand how the correspondence $\varphi\mapsto L(\varphi)$ behaves with
respect to the kernels. By putting $\varphi' = L(\varphi)$, we have that
\begin{equation*}
\begin{gathered}
K = \Ker\varphi = \{\sum_j\e_j f_j\mid f_j\in A, \sum_j g_j f_j = 0\}, \\
K' = \Ker\varphi' = \{\sum_j\e'_j f_j\mid f_j\in S, \sum_j g'_j f_j = 0\}.
\end{gathered}
\end{equation*}
Recall that the elements of the graded right submodule $K\subset \bigoplus_j A[-\Delta_j]$
are the right syzygies of the basis $\{g_j\}$. In a similar way, the elements
of the graded submodule $K'\subset \bigoplus_j S[-\Delta_j]$ are called the {\em syzygies
of the basis $\{g'_j\}$}. We will prove that $K'$ strictly contains the letterplace
analogue $L(K)$ in a meaningful way. Let us start with the following example.

\begin{Example}
Consider the graded right ideal $I = \langle x_1,\ldots,x_n \rangle\subset F$
and its letterplace analogue $J = L(I) = \langle x_{11},\ldots,x_{n1} \rangle\subset R$.
We have then the maps $\varphi:F[-1]^n\to F\,,\, \varphi(\e_j) = x_j$
and $\varphi' = L(\varphi):R[-1]^n\to R\,,\, \varphi'(\e'_j) = x_{j1}$.
One has clearly that $K = \Ker\varphi = 0$ and hence $L(K) = 0$. Moreover, from $R = P/Q$
it follows that
\[
K' = \Ker\varphi' = \langle \e'_j x_{k1}\mid 1\leq j,k\leq n \rangle.
\]
Clearly $K'$ is not a letterplace submodule but we have that $0 = \iota(K_d) = K'(d)_d =
K'\cap R(1,d)[-1]_d^n$, for all $d\geq 0$.
\end{Example}

From the above example we understand immediately that there is another set of natural
syzygies in $K'$. In fact, since $g'_j = \iota(g_j)\in\iota(A^r_{\Delta_j}) =
S(\Delta_j)^r_{\Delta_j}$ for any $1\leq j\leq s$, we have immediately that
\begin{equation}
\label{trsyz}
C = \langle \e'_j x_{kl}\mid 1\leq j\leq s, 1\leq k\leq n, 1\leq l\leq\Delta_j \rangle
\subset K'.
\end{equation}
Clearly, the given monomial basis of the graded submodule $C$ is a minimal one.
It is useful to consider the following property.

\begin{proposition}
\label{commdia2}
The following diagram is a commutative one
\[
\begin{array}{ccc}
\bigoplus_j A[-\Delta_j] & {\buildrel \varphi\over\longrightarrow} & A^r \\
\iota \downarrow & & \hspace{2pt} \downarrow \iota \\
\bigoplus_j S[-\Delta_j] & {\buildrel L(\varphi)\over\longrightarrow} & S^r \\
\end{array}
\]
\end{proposition}

\begin{proof}
If $\varphi' = L(\varphi)$ then one has to prove that $\iota(\varphi(h)) = \varphi'(\iota(h))$,
for all $h = \sum_j \e_j f_j\in \bigoplus_j A[-\Delta_j]$. In fact, by Proposition
\ref{iotapro} we have that
\[
\iota(\varphi(h)) = \iota(\sum_j g_j f_j) = \sum_j \iota(g_j)\sigma^{\Delta_j}(\iota(f_j)) =
\sum_j g'_j\sigma^{\Delta_j}(\iota(f_j)).
\]
On the other hand, one has immediately that
\[
\varphi'(\iota(h)) = \varphi'(\sum_j \e'_j \sigma^{\Delta_j}(\iota(f_j))) =
\sum_j g'_j \sigma^{\Delta_j}(\iota(f_j)).
\]
\end{proof}

\begin{theorem}
\label{kercorr2}
We have that $K' = C + L(K)$. In particular, one obtains that $b_d(K) = b_d(K') - b_d(C)$,
for all $d\geq 0$.
\end{theorem}

\begin{proof}
From Proposition \ref{commdia2} it follows that $L(K)\subset K'$ and we have already
observed that $C\subset K'$. Consider now $h' = \sum_j \e'_j f'_j$ any element of a
minimal basis of $K'$ where $f'_j\in S_{d-\Delta_j}$, for some $d\geq 0$
and for all $j$. Modulo $C\subset K'$, we can clearly assume that there is
$d'\geq d$ such that $f'_j\in S(\Delta_j,d')$ for each $j$. Moreover, since $h'$
is a minimal generator of $K'$ and $\sum_j g'_j f'_j = 0$, we may choose $d' = d$,
that is, $f'_j\in S(\Delta_j,d)_{d-\Delta_j}$ for all $j$. We conclude that
$f'_j = \sigma^{\Delta_j}(\iota(f_j))$ with $f_j\in A_{d-\Delta_j}$ and therefore
$h' = \iota(h)$ where $h = \sum_j \e_j f_j\in K$.

Finally, from the above argument it follows also that $b_d(K') =
b_d(C) + b_d(L(K))$ for any $d$, where $b_d(L(K)) = b_d(K)$ owing to
Theorem \ref{bascorr2}.
\end{proof}

Note that the graded Betti numbers $\{b_d(C)\}$ are immediately obtained in the following way.
For any $d\geq 0$, denote $m_d = \#\{ 1\leq j\leq s\mid \Delta_j = d \}$. We have clearly
that $b_0(C) = 0$ and for each $d\geq 0$
\[
b_{d+1}(C) = m_d\cdot n\cdot d.
\]
By Theorem \ref{bascorr2} and Theorem \ref{kercorr2}, one has therefore a letterplace
method to compute the couple of graded Betti numbers sets $\{b_d(M)\}$ and $\{b_d(K)\}$
of a graded right submodule $M\subset A^r$. In case one has that
$M\subset \bigoplus_i A[-\delta_i]$ for some integers $\delta_i\geq 0$, recall that
we can apply to $M$ the module component homogenization described in Section 3.
In particular, this homogenization is always needed for the kernel $K = \Ker\varphi\subset
\bigoplus_j A[-\Delta_j]$ if one wants to compute a minimal free right resolution
of $N = \bigoplus_i A[-\delta_i]/M$. In the next section we study an example to show
how this method works in practice.


\section{An illustrative example}

To the aim of illustrating the proposed methods by means of a concrete example,
let us fix the field $\QQ$ of rational numbers and consider the free
associative algebra in three variables $F = \QQ\langle x,y,z \rangle$. Then,
we define the graded two-sided ideal
\[
I = \langle [[a,b],c]\mid a,b,c\in \{x,y,z\} \rangle\subset F
\]
where by definition $[a,b] = a b - b a$. The corresponding quotient graded algebra
$A = F/I$ is the universal enveloping algebra of the free nilpotent Lie algebra
of class 2 which is freely generated by three variables. We want to compute the graded
homology of $A$, that is, a minimal free right resolution of its base field $\QQ$.
We start by considering the augmentation ideal
\[
M_1 = \langle x, y, z \rangle\subset A.
\]
By linear algebra, one can easily compute that a minimal basis of the ideal $I$
is given by the following (noncommutative) polynomials
\begin{equation}
\begin{gathered}
- [[z,y],z] = z^2y - 2zyz + yz^2,
- [[y,z],y] = zy^2 - 2yzy + y^2z, \\
- [[z,x],z] = z^2x - 2zxz + xz^2,
[[x,z],y] = yzx - zxy + xzy - yxz, \\
[[x,y],z] = zyx - zxy - yxz + xyz,
- [[y,x],y] = y^2x - 2yxy + xy^2, \\
- [[x,z],x] = zx^2 - 2xzx + x^2z,
- [[x,y],x] = yx^2 - 2xyx + x^2y.
\end{gathered}
\end{equation}
This immediately implies that the right syzygy module $M_2\subset A[-1]^3$ of the
minimal basis $\{x,y,z\}$ of $M_1$ has a minimal basis consisting of following
homogeneous elements
\begin{equation}
\begin{gathered}
\label{M2}
\e_3zy - 2\e_3yz + \e_2z^2,
\e_3y^2 - 2\e_2zy + \e_2yz, \\
\e_3zx - 2\e_3xz + \e_1z^2,
\e_2zx - \e_3xy + \e_1zy - \e_2xz, \\
\e_3yx - \e_3xy - \e_2xz + \e_1yz,
\e_2yx - 2\e_2xy + \e_1y^2, \\
\e_3x^2 - 2\e_1zx + \e_1xz,
\e_2x^2 - 2\e_1yx + \e_1xy.
\end{gathered}
\end{equation}
We need now to compute the right syzygy module $M_3\subset A[-3]^8$ of the above basis.
Our approach is based on Theorem \ref{kercorr2} and hence we start by defining the polynomial
algebra $P = \QQ[x_j,y_j,z_j\mid j\geq 1]$ and the quotient algebra $R = P/Q$ where
$Q = \langle a_jb_j\mid a_j,b_j\in \{x_j,y_j,z_j\},j\geq 1 \rangle$. Then, the letterplace
analogue $J = L(I)$ is defined as the graded $\sigma$-invariant ideal of $R$ which is
$\sigma$-generated by the (commutative) polynomials
\begin{equation}
\begin{gathered}
z_1z_2y_3 - 2z_1y_2z_3 + y_1z_2z_3,
z_1y_2y_3 - 2y_1z_2y_3 + y_1y_2z_3, \\
z_1z_2x_3 - 2z_1x_2z_3 + x_1z_2z_3,
y_1z_2x_3 - z_1x_2y_3 + x_1z_2y_3 - y_1x_2z_3, \\
z_1y_2x_3 - z_1x_2y_3 - y_1x_2z_3 + x_1y_2z_3,
y_1y_2x_3 - 2y_1x_2y_3 + x_1y_2y_3, \\
z_1x_2x_3 - 2x_1z_2x_3 + x_1x_2z_3,
y_1x_2x_3 - 2x_1y_2x_3 + x_1x_2y_3.
\end{gathered}
\end{equation}
Observe now that we cannot immediately apply Theorem \ref{kercorr2} because the elements
of the canonical basis of $A[-1]^3$ have degrees different from zero. Then, by the module
component homogenization of Section 3, we transform the basis (\ref{M2}) into the set
\begin{equation}
\label{barM2}
\begin{gathered}
\bee_3tzy - 2\bee_3tyz + \bee_2tz^2,
\bee_3ty^2 - 2\bee_2tzy + \bee_2tyz, \\
\bee_3tzx - 2\bee_3txz + \bee_1tz^2,
\bee_2tzx - \bee_3txy + \bee_1tzy - \bee_2txz, \\
\bee_3tyx - \bee_3txy - \bee_2txz + \bee_1tyz,
\bee_2tyx - 2\bee_2txy + \bee_1ty^2, \\
\bee_3tx^2 - 2\bee_1tzx + \bee_1txz,
\bee_2tx^2 - 2\bee_1tyx + \bee_1txy.
\end{gathered}
\end{equation}
which is a minimal basis of $\bM_2 = H(M_2)\subset \bA^3$ ($\bA = \bF/\bI$ where
$\bF = \QQ\langle x,y,z,t \rangle$ and $\bI$ is the extension of $I$ to $\bF$).
Then, we consider the letterplace analogue $\bM'_2 = L(\bM_2)\subset \bS^3$
($\bS = \bR/\bJ$ where $\bR$ corresponds to $\bF$ and $\bJ = L(\bI)$)
which has the minimal basis
\begin{equation}
\begin{gathered}
\bee'_3t_1z_2y_3 - 2\bee'_3t_1y_2z_3 + \bee'_2t_1z_2z_3,
\bee'_3t_1y_2y_3 - 2\bee'_2t_1z_2y_3 + \bee'_2t_1y_2z_3, \\
\bee'_3t_1z_2x_3 - 2\bee'_3t_1x_2z_3 + \bee'_1t_1z_2z_3,
\bee'_2t_1z_2x_3 - \bee'_3t_1x_2y_3 + \bee'_1t_1z_2y_3 - \bee'_2t_1x_2z_3, \\
\bee'_3t_1y_2x_3 - \bee'_3t_1x_2y_3 - \bee'_2t_1x_2z_3 + \bee'_1t_1y_2z_3,
\bee'_2t_1y_2x_3 - 2\bee'_2t_1x_2y_3 + \bee'_1t_1y_2y_3, \\
\bee'_3t_1x_2x_3 - 2\bee'_1t_1z_2x_3 + \bee'_1t_1x_2z_3,
\bee'_2t_1x_2x_3 - 2\bee'_1t_1y_2x_3 + \bee'_1t_1x_2y_3.
\end{gathered}
\end{equation}
In Section 9 we will show that one can algorithmically compute a minimal basis
of the corresponding syzygy module $\bM'_3\subset \bS[-3]^8$, namely
\begin{equation}
\begin{gathered}
\bee'_jx_k, \bee'_jy_k, \bee'_jz_k, \bee'_jt_k\ (1\leq j\leq 8, 1\leq k\leq 3), \\
\bee'_1y_4 + \bee'_2z_4,
\bee'_1x_4 + \bee'_3y_4 + \bee'_5z_4 - 2 \bee'_4z_4,
\bee'_2x_4 - 2 \bee'_5y_4 + \bee'_4y_4 + \bee'_6z_4, \\
\bee'_3x_4 + \bee'_7z_4,
\bee'_5x_4 + \bee'_4x_4 + \bee'_7y_4 + \bee'_8z_4,
\bee'_6x_4 + \bee'_8y_4, \\
\bee'_1z_4x_5 - \bee'_3z_4y_5 - \bee'_1x_4z_5 + 2\bee'_3y_4z_5 - \bee'_4z_4z_5, \\
\bee'_2z_4x_5 - \bee'_3y_4y_5 + 2\bee'_4z_4y_5 - \bee'_2x_4z_5 - \bee'_4y_4z_5, \\
\bee'_2y_4x_5 - 2\bee'_2x_4y_5 + \bee'_5y_4y_5 - \bee'_4y_4y_5 + \bee'_6z_4y_5 - \bee'_6y_4z_5, \\
\bee'_3y_4x_5 - \bee'_5z_4x_5 - \bee'_3x_4y_5 + 2\bee'_5x_4z_5 + \bee'_8z_4z_5, \\
\bee'_4y_4x_5 - \bee'_6z_4x_5 + \bee'_5x_4y_5 - \bee'_4x_4y_5 + \bee'_8z_4y_5 - \bee'_8y_4z_5, \\
\bee'_4x_4x_5 + 2\bee'_7y_4x_5 - \bee'_8z_4x_5 - \bee'_7x_4y_5 + \bee'_8x_4z_5.
\end{gathered}
\end{equation}
By Theorem \ref{kercorr2} one obtains that a minimal basis of the right syzygy module
$\bM_3\subset\bA[-3]^8$ of the minimal basis (\ref{barM2}) is given by the following
elements
\begin{equation}
\label{barM3}
\begin{gathered}
\bee_1y + \bee_2z,
\bee_1x + \bee_3y + \bee_5z - 2 \bee_4z,
\bee_2x - 2 \bee_5y + \bee_4y + \bee_6z, \\
\bee_3x + \bee_7z,
\bee_5x + \bee_4x + \bee_7y + \bee_8z,
\bee_6x + \bee_8y, \\
\bee_1zx - \bee_3zy - \bee_1xz + 2\bee_3yz - \bee_4z^2, \\
\bee_2zx - \bee_3y^2 + 2\bee_4zy - \bee_2xz - \bee_4yz, \\
\bee_2yx - 2\bee_2xy + \bee_5y^2 - \bee_4y^2 + \bee_6zy - \bee_6yz, \\
\bee_3yx - \bee_5zx - \bee_3xy + 2\bee_5xz + \bee_8z^2, \\
\bee_4yx - \bee_6zx + \bee_5xy - \bee_4xy + \bee_8zy - \bee_8yz, \\
\bee_4x^2 + 2\bee_7yx - \bee_8zx - \bee_7xy + \bee_8xz.
\end{gathered}
\end{equation}
From Theorem \ref{kercorr} it follows that $\bM_3 = H(M_3)$ and therefore a minimal
basis of $M_3$ is obtained simply by substituting in (\ref{barM3}) the canonical basis
$\{\bee_j\}$ of $\bA[-3]^8$ with the canonical basis $\{\e_i\}$ of $A[-3]^8$.

For computing now a minimal basis of the right syzygies module
$M_4\subset A[-4]^6\oplus A[-5]^6$ of the minimal basis of $M_3$, by the same approach
one has to consider the following minimal basis of a graded submodule $\bM'_3\subset \bS^8$
\begin{equation}
\begin{gathered}
\bee'_1t_1t_2t_3y_4 + \bee'_2t_1t_2t_3z_4,
\bee'_1t_1t_2t_3x_4 + \bee'_3t_1t_2t_3y_4 + \bee'_5t_1t_2t_3z_4 - 2 \bee'_4t_1t_2t_3z_4, \\
\bee'_2t_1t_2t_3x_4 - 2 \bee'_5t_1t_2t_3y_4 + \bee'_4t_1t_2t_3y_4 + \bee'_6t_1t_2t_3z_4,
\bee'_3t_1t_2t_3x_4 + \bee'_7t_1t_2t_3z_4, \\
\bee'_5t_1t_2t_3x_4 + \bee'_4t_1t_2t_3x_4 + \bee'_7t_1t_2t_3y_4 + \bee'_8t_1t_2t_3z_4,
\bee'_6t_1t_2t_3x_4 + \bee'_8t_1t_2t_3y_4, \\
\bee'_1t_1t_2t_3z_4x_5 - \bee'_3t_1t_2t_3z_4y_5 - \bee'_1t_1t_2t_3x_4z_5 + 2\bee'_3t_1t_2t_3y_4z_5
- \bee'_4t_1t_2t_3z_4z_5, \\
\bee'_2t_1t_2t_3z_4x_5 - \bee'_3t_1t_2t_3y_4y_5 + 2\bee'_4t_1t_2t_3z_4y_5 - \bee'_2t_1t_2t_3x_4z_5
- \bee'_4t_1t_2t_3y_4z_5, \\
\bee'_2t_1t_2t_3y_4x_5 - 2\bee'_2t_1t_2t_3x_4y_5 + \bee'_5t_1t_2t_3y_4y_5 - \bee'_4t_1t_2t_3y_4y_5
+ \bee'_6t_1t_2t_3z_4y_5 \\
-\ \bee'_6t_1t_2t_3y_4z_5, \\
\bee'_3t_1t_2t_3y_4x_5 - \bee'_5t_1t_2t_3z_4x_5 - \bee'_3t_1t_2t_3x_4y_5 + 2\bee'_5t_1t_2t_3x_4z_5
+ \bee'_8t_1t_2t_3z_4z_5, \\
\bee'_4t_1t_2t_3y_4x_5 - \bee'_6t_1t_2t_3z_4x_5 + \bee'_5t_1t_2t_3x_4y_5 - \bee'_4t_1t_2t_3x_4y_5
+ \bee'_8t_1t_2t_3z_4y_5 \\
-\ \bee'_8t_1t_2t_3y_4z_5, \\
\bee'_4t_1t_2t_3x_4x_5 + 2\bee'_7t_1t_2t_3y_4x_5 - \bee'_8t_1t_2t_3z_4x_5 - \bee'_7t_1t_2t_3x_4y_5
+ \bee'_8t_1t_2t_3x_4z_5.
\end{gathered}
\end{equation}
We can calculate that a minimal basis of the corresponding syzygy module
$\bM'_4\subset\bS[-4]^6\oplus\bS[-5]^6$ is given by the following elements
\begin{equation}
\begin{gathered}
\bee'_jx_k, \bee'_jy_k, \bee'_jz_k, \bee'_jt_k\ (1\leq j\leq 6, 1\leq k\leq 4), \\
\bee'_jx_k, \bee'_jy_k, \bee'_jz_k, \bee'_jt_k\ (7\leq j\leq 12, 1\leq k\leq 5), \\
\bee'_9x_6 + \bee'_3x_5y_6 - \bee'_6z_5y_6 + \bee'_{11}y_6 + \bee'_6y_5z_6,
\bee'_1z_5x_6 - \bee'_7y_6 - \bee'_1x_5z_6 - \bee'_8z_6, \\
\bee'_2z_5x_6 - \bee'_7x_6 - \bee'_4z_5y_6 - \bee'_2x_5z_6 + 2\bee'_4y_5z_6
- \bee'_5z_5z_6 + \bee'_{10}z_6, \\
\bee'_3z_5x_6 - \bee'_8x_6 - \bee'_4y_5y_6 + 2\bee'_5z_5y_6 - 2\bee'_{10}y_6
- \bee'_3x_5z_6 - \bee'_5y_5z_6 \\
+ \bee'_6z_5z_6 + \bee'_{11}z_6, \\
\bee'_1y_5x_6 - 2\bee'_1x_5y_6 + \bee'_2y_5y_6 + \bee'_3z_5y_6 + \bee'_8y_6
- \bee'_3y_5z_6 - \bee'_9z_6, \\
\bee'_4y_5x_6 - \bee'_{10}x_6 - \bee'_4x_5y_6 + \bee'_5x_5z_6 - \bee'_{12}z_6, \\
\bee'_5y_5x_6 - 2\bee'_6z_5x_6 - \bee'_{11}x_6 - \bee'_{12}y_6 + \bee'_6x_5z_6, \\
\bee'_1x_5x_6 - 2\bee'_2y_5x_6 - \bee'_3z_5x_6 - \bee'_8x_6 + \bee'_2x_5y_6
+ \bee'_{10}y_6 - \bee'_6z_5z_6 - 2\bee'_{11}z_6. \\
\end{gathered}
\end{equation}
We conclude that a minimal basis of $M_4$ is given by
\begin{equation}
\begin{gathered}
\label{M4}
\e_9x + \e_3xy - \e_6zy + \e_{11}y + \e_6yz,
\e_1zx - \e_7y - \e_1xz - \e_8z, \\
\e_2zx - \e_7x - \e_4zy - \e_2xz + 2\e_4yz
- \e_5z^2 + \e_{10}z, \\
\e_3zx - \e_8x - \e_4y^2 + 2\e_5zy - 2\e_{10}y
- \e_3xz - \e_5yz + \e_6z^2 + \e_{11}z, \\
\e_1yx - 2\e_1xy + \e_2y^2 + \e_3zy + \e_8y
- \e_3yz - \e_9z, \\
\e_4yx - \e_{10}x - \e_4xy + \e_5xz - \e_{12}z, \\
\e_5yx - 2\e_6zx - \e_{11}x - \e_{12}y + \e_6xz, \\
\e_1x^2 - 2\e_2yx - \e_3zx - \e_8x + \e_2xy
+ \e_{10}y - \e_6z^2 - 2\e_{11}z. \\
\end{gathered}
\end{equation}
By similar computations, we obtain a minimal basis of the right syzygy module
$M_5\subset A[-6]^8$ of the minimal basis (\ref{M4}) of $M_4$ which is
\begin{equation}
\begin{gathered}
\label{M5}
\e_2yx - \e_5zx - 2\e_2xy + \e_3y^2 + \e_4zy + 2\e_5xz + \e_8yz + \e_1z^2, \\
\e_2x^2 - 2\e_3yx - \e_8zx - \e_4zx + \e_3xy - \e_6zy + \e_8xz + 2\e_6yz - \e_7z^2, \\
\e_5x^2 + \e_8yx + \e_4yx + 2\e_1zx - \e_4xy + \e_6y^2 - 2\e_7zy - \e_1xz + \e_7yz.
\end{gathered}
\end{equation}
Finally, a minimal basis of the right syzygy module $M_6\subset A[-8]^3$ of the above
minimal basis can be calculated as consisting of one single free element
\begin{equation}
\label{M6}
\e_1x + \e_2y - \e_3z.
\end{equation}
All the computed minimal bases define therefore a finite minimal free right resolution
of the base field $\QQ\simeq A/M_1$ of the algebra $A$.
Such resolution reads
\begin{equation}
\label{resol}
\begin{gathered}
0 \leftarrow \QQ\leftarrow A\leftarrow A[-1]^3\leftarrow A[-3]^8\leftarrow
A[-4]^6\oplus A[-5]^6 \leftarrow A[-6]^8 \\
\leftarrow A[-8]^3\leftarrow A[-9]\leftarrow 0.
\end{gathered}
\end{equation}
This exact sequence describes the homology of the graded algebra $A$ since one has
that $b_d(M_i) = \dim_\QQ \Tor_{i,d}^A(\QQ,\QQ)$ ($1\leq i\leq 6, d\geq 0$).
In particular, the corresponding graded Betti numbers table is
\[
\begin{tabular}{c|ccccccc}
  & 0 & 1 & 2 & 3 & 4 & 5 & 6 \\
\hline
0 & 1 & 3 & - & - & - & - & - \\
1 & - & - & 8 & 6 & - & - & - \\
2 & - & - & - & 6 & 8 & - & - \\
3 & - & - & - & - & - & 3 & 1 \\
\end{tabular}
\]
Note that this is a standard way to represent such table. Recall that if a syzygy
belongs to the kernel of $(i+1)$-th map in the resolution then $i$ is called the
{\em homological degree} of the syzygy. Moreover, the (induced) degree of a syzygy
is also called the {\em internal degree} and the {\em slanted degree} is by definition
the difference between the internal and homological degree.
In the above graded Betti numbers table the columns are then indexed by the
homological degrees and the rows are indexed by the slanted degrees. By such table
we conclude that the Castelnuovo-Mumford regularity of $A$ is 3 and its global
(homological) dimension is 6 (see, for instance, \cite{Uf}). Observe that
the finiteness of these numbers is due to the property of the algebra $A$
to have a PBW basis, that is, $A$ is a G-algebra \cite{LeS}. Another reason
for the right syzygy modules $M_i$ ($1\leq i\leq 6$) to be finitely generated
is that the ideal $I$ has a finite \Gr\ basis \cite{DLS} with respect to
the graded lexicographic monomial ordering of $F$ and hence the corresponding
Anick's resolution \cite{An,Uf} of $A = F/I$ consists of a finite number of chains
for each homological degree. Note that such resolution is generally not of finite
length or minimal but in fact in this case it coincides with the finite minimal
resolution (\ref{resol}).


\section{Finite computations}

In this section we explain how the computation of noncommutative resolutions,
as the one that we have just illustrated, can be obtained in an algorithmic way.
Since our approach is to reduce such calculations to analogous ones for modules
over polynomial algebras in commutative variables, the problem of being algorithmic
essentially consists in working only with a finite number of such variables.
We start by analyzing in general the amount of right syzygies that can be
obtained with a finite number of letterplace variables. Then, we will show that
there are some cases when a suitable large number of them provides the complete
computation of a right syzygy module and iteratively of a finite number of
such modules in a minimal right resolution.

Let $M' = L(M)\subset S^r$ be the letterplace analogue of a finitely generated graded
right submodule $M\subset A^r$. Recall that $A = F/I$ and $S = R/J$ where $J = L(I)$. 
With the notation of Section 4, for all integers $d\geq 0$ we have that $P(d) =
\KK[x_{ij}\mid 1\leq i\leq n, 1\leq j\leq d]$ and $R(d) = P(d)/Q(d)$ where
$Q(d) = P(d)\cap Q$. Moreover, we put $J(d) = J\cap R(d)$ and $S(d) = R(d)/J(d)$.
We finally define $M'(d) = M'\cap S(d)^r$. Assume now that $M'$ is generated by $M'(d)$.
If $\{g_j\}$ ($1\leq j\leq s$) is a minimal basis of $M$ then this happens when
the maximal degree in $\{g_j\}$ is bounded by $d$ and therefore the minimal basis
$\{g'_j\}$ of $M'$ ($g'_j = \iota(g_j)$) is in fact contained in $M'(d)$.
Let $K$ be the right syzygy module of $\{g_j\}$ and let $K'$ be the corresponding syzygy
module of $\{g'_j\}$. Put $\Delta_j = \deg(g_j) = \deg(g'_j)\leq d$.
By Theorem \ref{kercorr2} we know that $K' = C + L(K)$ where by definition
$C = \langle \e'_j x_{kl}\mid 1\leq j\leq s, 1\leq k\leq n, 1\leq l\leq\Delta_j \rangle$.
By abuse of notation, we denote also by $\{\e'_j\}$ the canonical basis
of $\bigoplus_j S(d)[-\Delta_j]$ and consider the graded right $S(d)$-module
homomorphism
\[
\varphi'_{\rst{d}}:\bigoplus_j S(d)[-\Delta_j]\to S(d)^r, \e'_j\mapsto g'_j.
\]
Define $K'_{\rst{d}} = \Ker\varphi'_{\rst{d}}$ the syzyzy module of the minimal
basis $\{g'_j\}$ of $M'(d)$. In other words, one has that
$K'_{\rst{d}} = \{\sum_j \e'_j f'_j\mid f'_j\in S(d), \sum_j g'_j f'_j = 0\}$ and
therefore
\[
K'_{\rst{d}} = K'\cap \bigoplus_j S(d)[-\Delta_j].
\]
Note that $C\subset K'_{\rst{d}}$ since $\Delta_j\leq d$. By Theorem \ref{kercorr2}
one obtains immediately what follows.

\begin{proposition}
\label{finsyz}
Let $B\cup\{h'_k\}$ be a minimal basis of $K'_{\rst{d}}$ where $B$ is the given
basis of the submodule $C$. If $h_k = \iota^{-1}(h'_k)$ then $\{h_k\}$ is a minimal
basis of $K$ for all degrees $\leq d$.
\end{proposition}

Since $P(d)$ is a polynomial algebra in a finite number of commutative
variables, observe that a minimal basis of the syzygy module $K'_{\rst{d}}$
can be computed algorithmically by any (commutative) computer algebra system
(see, for instance, \cite{DGPS}). By iterating this method, it is then clear that
one obtains the partial computation of a minimal free right resolution of $N = A^r/M$
up to degree $d$.

\medskip
We may ask now when $K'$ is generated by $K'_{\rst{d}}$, that is, $K$ is finitely
generated in degrees $\leq d$ and hence a complete minimal basis of $K$ can be
obtained by computing a minimal basis of $K'_{\rst{d}}$. This is a complicated issue
because we have already observed that finitely presented noncommutative algebras
are generally not right Noetherian (or even right coherent) and hence right syzygy
modules are not always finitely generated.

Since the computation of syzygies is intimately related to the notion of \Gr\ basis
it is not surprising that to provide some cases when the right syzygy module is
finitely generated we consider the monomial case. Let $I\subset F$ be a finitely
generated monomial two-sided ideal and consider the finitely presented monomial algebra
$A = F/I$. If $M\subset \bigoplus_i A[-\delta_i]$ is a finitely generated monomial
right submodule then we want to analyze the right syzygy module of a monomial basis
of $M$. For the sake of simplicity, we start with the simplest case, that is,
$M\subset A$ is a cyclic right ideal. Recall that $W$ is the set of monomials of $F$
and let $w\in W\setminus I$. Assume that $M$ is generated by the coset $\bw = w + I$.
Put $\Delta = \deg(\bw) = \deg(w)$ and consider the graded right $A$-module
homomorphism
\[
\varphi:A[-\Delta]\to A\,,\, 1\mapsto \bw.
\]
If $K = \Ker\varphi$ then one has immediately that $K = (I :_R w)/I$ where by definition
\[
(I :_R w) = \{f\in F\mid w f\in I\}.
\]
The above set is a right ideal of $F$ which is called the {\em right colon ideal
of $I$ with respect to $w$} (see, for instance, \cite{Mo}). Because $I$ is a monomial
two-sided ideal and $w\in W$, we have that $I\subset (I :_R w)$ which is a monomial
right ideal. Moreover, the assumption $w\notin I$ is equivalent to
$(I :_R w)\neq \langle 1\rangle$. A complete description of the right ideal $(I :_R w)$
is provided by the following result.

\begin{proposition}[\cite{LS2}, Proposition 4.10]
Let $\{v_k\}\subset W$ be a monomial basis of $I$ and put $d_k = \deg(v_k)$.
For all $k$, we define the finitely generated monomial right ideal
\[
I_w(v_k) = \langle u_{kl}\mid w u_{kl} = t_{kl} v_k, t_{kl},u_{kl}\in W,
\deg(u_{kl}) < d_k \rangle.
\]
Then, one has that $(I :_R w) = \sum_k I_w(v_k) + I$.
\end{proposition}

Since we are assuming that $\{v_k\}$ is a finite set and the right annihilator
ideal $K = \Ker\varphi\subset A[-\Delta]$ of the coset $\bw = w + I\in A$
($\Delta = \deg(w)$) is generated by the cosets modulo $I$ of the elements
of a monomial basis of $\sum_k I_w(v_k)$, one obtains immediately what follows.

\begin{lemma}
\label{cycdeg}
The right ideal $K\subset A[-\Delta]$ is finitely generated and monomial.
Precisely, if $d = \max\{d_k\}$ is the maximal degree of a finite monomial
basis of $I$ then the maximal (induced) degree of a minimal monomial basis
of $K$ is bounded by $\Delta + d - 1$.
\end{lemma}

Let now $M = \langle e_i\bw_{ij} \rangle\subset \bigoplus_i A[-\delta_i]$ be a finitely
generated monomial right submodule where $\bw_{ij} = w_{ij} + I$ and
$w_{ij}\in W\setminus I$. If $\Delta_{ij} = \deg(e_i\bw_{ij}) = \delta_i + \deg(w_{ij})$
and $\{\e_{ij}\}$ is the canonical basis of $\bigoplus_{i,j} A[-\Delta_{ij}]$ then
we consider the graded right $A$-module homomorphism
\[
\varphi:\bigoplus_{i,j} A[-\Delta_{ij}]\to \bigoplus_i A[-\delta_i]
\,,\, \e_{ij}\mapsto e_i\bw_{ij}.
\]
One has clearly that $K = \Ker\varphi = \bigoplus_{i,j} \e_{ij} (I :_R w_{ij})/I$
which implies the following result.

\begin{proposition}
The right syzygy module $K\subset\bigoplus_{i,j} A[-\Delta_{ij}]$ is finitely generated
and monomial. Precisely, if $d$ is the maximal degree of a finite monomial basis
of $I$ and $\Delta = \max\{\Delta_{ij}\}$ then the maximal (induced) degree
of a minimal monomial basis of $K$ is bounded by $\Delta + d - 1$.
\end{proposition}

By applying iteratively the above result we obtain what follows.

\begin{theorem}
\label{monbound}
Let $A = F/I$ be a finitely presented monomial algebra and let
$M\subset \bigoplus_{1\leq i\leq r} A[-\delta_i]$ be a finitely generated monomial
submodule. There exists a minimal free right resolution of
$N = \bigoplus_{1\leq i\leq r} A[-\delta_i]/M$ where all right syzygies modules $M_i$
are finitely generated and monomial. Moreover, if $b_k(I) = 0$ for $k > d$ and
$b_k(M) = 0$ for $k > \Delta$ then $b_k(M_i) = 0$, for all $k > \Delta + (i-1)(d-1)$.
Note finally that such resolution is not necessarily of finite length.
\end{theorem}

A major application of the above result is when $M = \langle x_1,\ldots,x_n \rangle\subset A$
is the augmentation ideal and hence $b_k(M_i) = \dim_\KK \Tor_{i,k}^A(\KK,\KK)$. In this case
we have that $\Delta = 1$ and an explicit combinatorial description of a minimal free right
resolution of $\KK\simeq A/M$ was introduced by Backelin \cite{Ba} in terms of the
monomial relations of $A$. The Anick's resolution is a (usually non-minimal) extension
of that resolution to general algebras by means of a \Gr\ basis of the ideal of relations.
A by-product of these constructions is that the homology of monomial algebras bounds
the homology of corresponding general algebras.
Precisely, fix a monomial ordering of $F$, that is, a multiplicatively compatible well-ordering
of $W$. If $0\neq f\in F$ then we denote by $\lm(f)\in W$ the greatest among the monomials
of $f$ with respect to such ordering. Let $I\subset F$ be a two-sided ideal and define
the monomial two-sided ideal $\LM(I) = \langle \lm(f)\mid f\in I, f\neq 0 \rangle$.
By definition, a {\em \Gr\ basis} of $I$ is a subset $\{g_j\}\subset I$ such that
$\{\lm(g_j)\}$ is a (monomial) basis of $\LM(I)$. Assume now that $I$ is a graded ideal
and consider the finitely generated graded algebras $A = F/I$ and $\hat{A} = F/\LM(I)$. 

\begin{theorem}[\cite{An}, Lemma 3.4]
\label{genbound}
For any homological degree $i$ and internal degree $k$, one has that
$\dim_\KK \Tor_{i,k}^A(\KK,\KK)\leq \dim_\KK \Tor_{i,k}^{\hat{A}}(\KK,\KK)$.
\end{theorem}

The above result together with Theorem \ref{monbound} provides that if $I$ has a finite
\Gr\ basis, that is, $\hat{A}$ is a finitely presented monomial algebra then $A$ has a
(possibly infinite) minimal free right resolution of the base field $\KK$
where all right syzygies modules $M_i$ are finitely generated. These results imply
also a bound $d$ for the maximal degree in a finite number of minimal bases in the
resolution. Then, by Proposition \ref{finsyz} we conclude that the proposed method
is able to fully compute such bases by working with modules over the finitely
generated (hence Noetherian) polynomial algebra $P(d)$.


\section{Tips and tricks}

Some additional tricks may be used to speedup the computation of a minimal right
resolution by our approach. Especially for higher syzygies, one may have that
the elements of the canonical basis of the free right module containing such syzygies
have large degrees. Let us denote by $d$ the minimum of these degrees. Owing to
the module component homogenization we have clearly that the corresponding letterplace
syzygies are all multiple of the monomial $t_1\cdots t_d$ and the letterplace variables
$x_{ij}$ occuring in these elements have the index $j > d$. For instance, in the example
of Section 8 one has that the letterplace encoding (after module component homogenization)
of the minimal basis (\ref{M4}) of the right syzygy module $M_4$ is
\begin{equation*}
\begin{gathered}
t_1t_2t_3t_4 (\bee'_9t_5x_6 + \bee'_3x_5y_6 - \bee'_6z_5y_6 + \bee'_{11}t_5y_6
+ \bee'_6y_5z_6), \\
t_1t_2t_3t_4 (\bee'_1z_5x_6 - \bee'_7t_5y_6 - \bee'_1x_5z_6 - \bee'_8t_5z_6), \\
t_1t_2t_3t_4 (\bee'_2z_5x_6 - \bee'_7t_5x_6 - \bee'_4z_5y_6 - \bee'_2x_5z_6 + 2\bee'_4y_5z_6
- \bee'_5z_5z_6 + \bee'_{10}t_5z_6), \\
t_1t_2t_3t_4 (\bee'_3z_5x_6 - \bee'_8t_5x_6 - \bee'_4y_5y_6 + 2\bee'_5z_5y_6 - 2\bee'_{10}t_5y_6
- \bee'_3x_5z_6 - \bee'_5y_5z_6 \\
+\ \bee'_6z_5z_6 + \bee'_{11}t_5z_6), \\
t_1t_2t_3t_4 (\bee'_1y_5x_6 - 2\bee'_1x_5y_6 + \bee'_2y_5y_6 + \bee'_3z_5y_6 + \bee'_8t_5y_6
- \bee'_3y_5z_6 - \bee'_9t_5z_6), \\
t_1t_2t_3t_4 (\bee'_4y_5x_6 - \bee'_{10}t_5x_6 - \bee'_4x_5y_6 + \bee'_5x_5z_6
- \bee'_{12}t_5z_6), \\
t_1t_2t_3t_4 (\bee'_5y_5x_6 - 2\bee'_6z_5x_6 - \bee'_{11}t_5x_6 - \bee'_{12}t_5y_6
+ \bee'_6x_5z_6), \\
t_1t_2t_3t_4 (\bee'_1x_5x_6 - 2\bee'_2y_5x_6 - \bee'_3z_5x_6 - \bee'_8t_5x_6 + \bee'_2x_5y_6
+ \bee'_{10}t_5y_6 - \bee'_6z_5z_6 \\
-\ 2\bee'_{11}t_5z_6). \\
\end{gathered}
\end{equation*}
For this example we have hence that $d = 4$. It is clear now that if we factor out
the monomial $t_1\cdots t_d$ and shift back the letterplace variables by $d$ then
we obtain again a minimal basis whose syzygies imply the next right syzygies that
one needs for computing the resolution. For instance, the simplified minimal basis
\begin{equation*}
\begin{gathered}
\bee'_9t_1x_2 + \bee'_3x_1y_2 - \bee'_6z_1y_2 + \bee'_{11}t_1y_2
+ \bee'_6y_1z_2, \\
\bee'_1z_1x_2 - \bee'_7t_1y_2 - \bee'_1x_1z_2 - \bee'_8t_1z_2, \\
\bee'_2z_1x_2 - \bee'_7t_1x_2 - \bee'_4z_1y_2 - \bee'_2x_1z_2 + 2\bee'_4y_1z_2
- \bee'_5z_1z_2 + \bee'_{10}t_1z_2, \\
\bee'_3z_1x_2 - \bee'_8t_1x_2 - \bee'_4y_1y_2 + 2\bee'_5z_1y_2 - 2\bee'_{10}t_1y_2
- \bee'_3x_1z_2 - \bee'_5y_1z_2 + \bee'_6z_1z_2 \\
+\ \bee'_{11}t_1z_2, \\
\bee'_1y_1x_2 - 2\bee'_1x_1y_2 + \bee'_2y_1y_2 + \bee'_3z_1y_2 + \bee'_8t_1y_2
- \bee'_3y_1z_2 - \bee'_9t_1z_2, \\
\bee'_4y_1x_2 - \bee'_{10}t_1x_2 - \bee'_4x_1y_2 + \bee'_5x_1z_2
- \bee'_{12}t_1z_2, \\
\bee'_5y_1x_2 - 2\bee'_6z_1x_2 - \bee'_{11}t_1x_2 - \bee'_{12}t_1y_2
+ \bee'_6x_1z_2, \\
\bee'_1x_1x_2 - 2\bee'_2y_1x_2 - \bee'_3z_1x_2 - \bee'_8t_1x_2 + \bee'_2x_1y_2
+ \bee'_{10}t_1y_2 - \bee'_6z_1z_2 - 2\bee'_{11}t_1z_2
\end{gathered}
\end{equation*}
has the following minimal syzygy basis
\begin{equation*}
\begin{gathered}
\bee'_jx_k, \bee'_jy_k, \bee'_jz_k, \bee'_jt_k\ (1\leq j\leq 8, 1\leq k\leq 2), \\
\bee'_2y_3x_4 - \bee'_5z_3x_4 - 2\bee'_2x_3y_4 + \bee'_3y_3y_4 + \bee'_4z_3y_4
+ 2\bee'_5x_3z_4 + \bee'_8y_3z_4 + \bee'_1z_3z_4, \\
\bee'_2x_3x_4 - 2\bee'_3y_3x_4 - \bee'_8z_3x_4 - \bee'_4z_3x_4 + \bee'_3x_3y_4
- \bee'_6z_3y_4 + \bee'_8x_3z_4 + 2\bee'_6y_3z_4 \\
-\ \bee'_7z_3z_4, \\
\bee'_5x_3x_4 + \bee'_8y_3x_4 + \bee'_4y_3x_4 + 2\bee'_1z_3x_4 - \bee'_4x_3y_4
+ \bee'_6y_3y_4 - 2\bee'_7z_3y_4 - \bee'_1x_3z_4 \\
+\ \bee'_7y_3z_4.
\end{gathered}
\end{equation*}
It is clear that from these elements one obtains the minimal right syzygy basis (\ref{M5}).
This trick generally speedups the computation because it reduces the number of needed
letterplace variables which is usually fixed in advance in a computer algebra system.
For instance, the minimal free right resolution of the example in Section 8
can be computed in 1 sec by applying this trick and just standard routines of \textsf{Singular}
\cite{DGPS} for computing commutative syzygies. Note that this timing is in fact comparable
with 0.1 sec that we have obtained with the optimized library \textsf{PLURAL} \cite{LeS}
for computations with G-algebras. We believe that this indicates that an effective
implementation of our method would be feasible also for algebras which have no PBW basis,
that is, that cannot be directly treated by computations over commutative variables.
Note that the computational experiment has been performed on a laptop running Singular 4.0.3
with a four core Intel i3 at 2.20GHz and 16 GB RAM.

Since the computation of commutative syzygies which is needed in Proposition \ref{finsyz}
is usually a by-product of a \Gr\ basis calculation, another usable trick we may finally
suggest is to use the monomial syzygies generating the submodule $C$ as criteria to avoid
useless S-polynomials. For more details about this idea to use syzygies as criteria
we refer to \cite{MMT}.


\section{Conclusions and future directions}

In this paper we have shown that by extending the notion of letterplace correspondence
to graded right modules over graded (noncommutative) algebras, one obtains a method
for computing minimal free right resolutions of such modules. By bounding the degree
of the right syzygies that one wants to compute in the resolution, this method results
in a feasible algorithm which requires only the computation of syzygies in (Noetherian)
modules over finitely generated commutative algebras. Such calculations are provided
by any (commutative) computer algebra system and hence our approach extends
the availability of homological computations to general graded algebras which are
not covered by ad hoc methods.

Of course, to improve the practical efficiency of the proposed algorithms it is necessary
to develop an implementation of them in the kernel of a computer algebra system using,
for instance, the tricks that we have explained in Section 10. As a further research
direction, we can suggest to extend the letterplace correspondence to graded bimodules
in order to obtain also Hochschild homology computations. Finally, we mention that
free resolutions of nongraded modules are of course worthwhile of investigation
and the letterplace correspondence, as proved in \cite{LS1}, can be suitably extended
to the nongraded case.

\section{Acknowledgments}

We would like to thank the anonymous referee for his/her careful reading of the
manuscript. We have sincerely appreciated all valuable comments and suggestions
as they have significantly improved the readability of the paper.


\end{document}